\documentclass[11pt,bezier]{article}
\usepackage{amsmath, amssymb, amsfonts, graphicx}

\newtheorem{prethm}{{\bf Theorem}}

\newenvironment{thm}{\begin{prethm}\sl{\hspace{-0.5
               em}{\bf.}}}{\end{prethm}}

\newtheorem{prepro}[prethm]{{\bf Proposition}}

\newtheorem{prelem}[prethm]{{\bf Lemma}}

\newenvironment{lem}{\begin{prelem}\sl{\hspace{-0.5
               em}{\bf.}}}{\end{prelem}}

\newtheorem{predeff}[prethm]{{\bf Definition}}

\newtheorem{precor}[prethm]{{\bf Corollary}}

\newtheorem{preconj}[prethm]{{\bf Conjecture}}

\newenvironment{conj}{\begin{preconj}\sl{\hspace{-0.5
               em}{\bf.}}}{\end{preconj}}

\newtheorem{preremark}[prethm]{{\bf Remark}}

\newenvironment{remark}{\begin{preremark}\rm{\hspace{-0.5
               em}{\bf.}}}{\end{preremark}}

\newtheorem{preexample}[prethm]{{\bf Example}}

\newtheorem{preproof}{{\bf\textsf{Proof.}}}

\newenvironment{proof}[1]{\begin{preproof}{\rm
               #1}\hfill{$\Box$}}{\end{preproof}}

\newcommand{\la}{\lambda}
\newcommand{\om}{\omega}
\newcommand{\x}{{\bf x}}

\newcommand{\mul}{{\rm mult}}

\title{Some spectral properties of  chain graphs}

\author{{\sc Ebrahim Ghorbani} \\[.3cm]
{\sl Department of Mathematics, K.N. Toosi University of Technology,}\\
{\sl P. O. Box 16315-1618, Tehran, Iran}\\
{\sl School of Mathematics, Institute for Research in Fundamental
Sciences (IPM),}\\
{\sl P.O. Box
19395-5746, Tehran, Iran }
\\[.3cm]
$\mathsf{e\_ghorbani@ipm.ir}$ }

\begin{document}
\maketitle

\vspace{5mm}

\begin{abstract}
A graph is called a chain graph if it is bipartite and the neighborhoods of the vertices in each color class form a chain with respect to inclusion. Alazemi,  Andeli\'c and  Simi\'c conjectured that no chain graph shares a non-zero (adjacency) eigenvalue with its vertex-deleted subgraphs.
We disprove this conjecture.
However, we show that the assertion holds for subgraphs obtained by deleting vertices of maximum degrees in either of color classes.
 We also give a simple proof for the fact that chain graphs have no eigenvalue in the interval $(0,1/2)$.

\vspace{5mm}
\noindent {\bf AMS Mathematics Subject Classification\,(2010):}   05C50
\end{abstract}

\vspace{5mm}

\section{Introduction}

A graph is a called a {\em  chain graph} (or {\em double nested graph} \cite{bcrs}) if it is bipartite and the neighborhoods of the vertices in each color class form a chain with respect to inclusion. Chain graphs appear in different contexts and so several characterizations of them can be found in the literature.
Here we mention a few: a graph $G$ is a chain graph if and only if it satisfies one of the following properties:
\begin{itemize}
\item every vertex $v_i$ of $G$ can be assigned a real number $a_i$ for which there exists a positive real number $R$ such that $|a_i | < R $ for all $i$ and  two vertices $v_i,v_j$ are adjacent if and only if $|a_i -a_j|\ge R$ (due to this property chain graphs are also called {\em difference graphs}) \cite{hpx};
  \item $G$ is a bipartite graph  and every induced
subgraph with no isolated vertices has a dominating vertex on each color class, that is, a vertex adjacent to all the vertices of the other color class \cite{hpx};
\item $G$ is $(2K_2,C_5,C_3)$-free;
  \item $G$ is $2K_2$-free and bipartite;
  \item $G$ is $P_5$-free and bipartite.
\end{itemize}
Note that the last three characterizations follow easily from the second one.

In terms of graph eigenvalues, (connected) chain graphs have a remarkable feature. They are characterized as graphs whose largest eigenvalue  is maximum  among the connected bipartite graphs with the same number of vertices and edges (\cite{bcrs,bfp}). Another family  with similar properties as chain graphs are {\em threshold graphs} which are the graphs such that the neighborhoods of their vertices form a single chain with respect to inclusion. They have the largest maximum eigenvalue among the graphs with prescribed number of vertices and edges (see \cite[Remarks~8.1.9]{crs}).  In fact, any threshold graph can be obtained from a chain graph $G$ by replacing one color class of $G$ by a clique, and all other edges unchanged.  For more information see \cite{bls,mp}.

 Alazemi,  Andeli\'c and  Simi\'c \cite{aas} conjectured that no chain graph shares a non-zero (adjacency) eigenvalue with its vertex-deleted subgraphs.
We disprove this conjecture. However, we show that the assertion holds for subgraphs obtained by deleting vertices of maximum degrees in either of color classes.
They \cite{aas} also proved that chain graphs have no eigenvalue in the interval $(0,1/2)$. We give a simple proof for this result.

\section{Preliminaries}\label{pre}

The graphs we consider are all simple and undirected.
For a  graph $G$, we denote  by $V(G)$ the vertex set of $G$. 
For two vertices $u,v$, by $u\sim v$ we mean that $u$ and $v$ are adjacent.
 If  $V(G)=\{v_1, \ldots , v_n\}$, then the {\em adjacency matrix} of $G$ is an $n \times  n$
 matrix $A(G)$ whose $(i, j)$-entry is $1$ if $v_i\sim v_j$ and  $0$ otherwise.
 By {\em eigenvalues}  of $G$ we mean those of $A(G)$.
 The multiplicity of an eigenvalue $\la$ of $G$ is denoted by $\mul(\la,G)$.
For a vertex $v$ of $G$, let $N(v)$ denote the {\em neighborhood} of $v$, i.e.   the set of
all vertices of $G$ adjacent to $v$.
Two vertices $u$ and $v$ of $G$ are called {\em duplicate} if $N(u)=N(v)$.
For $v\in V(G)$, we use the notation $G-v$ to mean the subgraph of $G$ induced by $V(G)\setminus \{v\}$.

\begin{remark}\label{struc} ({\em Structure of chain graphs}) As it was observed in \cite{bcrs}, the color classes of any chain graph $G$ can be partitioned into $k$ non-empty cells $U_1,\ldots, U_k$ and $V_1,\ldots, V_k$  such that $$N(u)=V_1\cup\cdots\cup V_{k-i+1}~~\hbox{for any}~ u\in U_i,~1\le i\le k.$$
\end{remark}

\begin{remark}\label{rem} ({\em Sum rule}) Let $\x$ be an eigenvector for eigenvalue $\la$ of a graph $G$. Then the entries of $\x$ satisfy the following equalities:
\begin{equation}\label{sumrule}
\la\x(v)=\sum_{u:\,u\sim v}\x(u),~~\hbox{for all}~v\in V(G).
\end{equation}
From this it is seen that if $\la\ne0$ and $N(v)=N(v')$, then $\x(v)=\x(v')$.
 In particular if $G$ is a chain graph, in the notations of Remark~\ref{struc}, $\x$ is constant on each $U_i$ and on each $V_i$ for $i=1,\ldots,k$.
\end{remark}

We will make use of the interlacing property of graph eigenvalues which we recall below (see \cite[Theorem~2.5.1]{bh}).
\begin{lem}\label{inter}
Let $G$ be a graph of order $n$, $H$ be an induced subgraph of $G$ of order $m$, $\la_1\ge\cdots\ge\la_n$ and $\mu_1\ge\cdots\ge\mu_m$ be the eigenvalues of $G$ and $H$, respectively. Then $$\la_i\ge\mu_i\ge\la_{n-m+i}~~\hbox{for}~ i=1, \ldots,m.$$
In particular, if $m=n-1$, then
 $$\la_1\ge\mu_1\ge\la_2\ge\mu_2\ge\cdots\ge\la_{n-1}\ge\mu_{n-1}\ge\la_n.$$
\end{lem}
From the case of equality in interlacing (see \cite[Theorem~2.5.1]{bh}) the following can be deduced.
\begin{lem}\label{eqinter} If in Lemma~\ref{inter}, we have $\la_i=\mu_i$ or $\mu_i=\la_{n-m+i}$ for some $1\le i\le m$, then $A(H)$ has an eigenvector $\x$ for $\mu_i$, such that  $\begin{pmatrix}\bf0 \\ \x\end{pmatrix}$, with the $\bf0$ vector corresponding to $V(G)\setminus V(H)$, is an eigenvector of $A(G)$ for the eigenvalue $\mu_i$.
\end{lem}

\section{Eigenvectors and downer vertices}
For a graph $G$ and an eigenvalue $\la$ of $G$, a vertex $v$ is called {\em downer} if $\mul(\la,G-v)=\mul(\la,G)-1$.
In \cite{aacd} it was shown that all the non-zero eigenvalues of chain graphs are simple (this also readily follows from (the proof of) Theorem~\ref{chain} below).
As the subgraphs of any chain graph are also chain graphs, if  $\la$ is an eigenvalue of a chain graph $G$, then removal of any vertex from $G$ does not increase the multiplicity of $\la$, i.e. $\mul(\la, G-v)\le\mul(\la, G-v)=1$. A question raises on the precise value of $\mul(\la,G-v)$: is it always 0?
This was actually conjectured in \cite{aas}.

\begin{conj}\label{conj} {\rm (\cite{aas})} In any chain graph, every vertex is downer with respect to every non-zero eigenvalue.
\end{conj}
The  conjecture is equivalent to say that for any chain graph $G$ and any $v\in V(G)$, $G-v$ shares no non-zero eigenvalue with $G$.

We disprove Conjecture~\ref{conj} in this section. Indeed, Theorems~\ref{1} and \ref{om} below show that there are infinitely many counterexamples for this conjecture. In spite of that, a weak version of the conjecture is true: in Theorem~\ref{chain} it will be shown that for non-zero eigenvalues the vertices with maximum degrees in each color class of a chain graph are downer.

\begin{remark}\label{downer}
For a vertex $v$ being downer or not depends on the component corresponding to $v$ in the eigenvectors of $\la$. Let $W$ be the eigenspace corresponding to $\la$. If for all $\x\in W$, we have $\x(v)=0$, then $v$ cannot be a downer vertex as for any $\x\in W$, the vector $\x'$
obtained by eliminating the the component corresponding to $v$, is an eigenvector of $\la$ for $G-v$, so we have
  $$\mul(\la,G-v)\ge\dim\,\{\x':\x\in W\}=\dim W=\mul(\la,G).$$
From this and Lemma~\ref{eqinter} it follows that,  in the case that $\mul(G,\la)=1$, there exists an eigenvector $\x$ for $\la$ with $\x(v)=0$ if and only if $v$ is not a downer vertex for $\la$.
\end{remark}

\begin{thm}\label{chain} Let $G$ be a chain graph. Then the vertices having maximum degrees in each color class of $G$ are downer  for any non-zero eigenvalue.
\end{thm}
\begin{proof}{In the notations of Remark~\ref{struc}, the vertices in $U_1$ and $V_1$ have the maximum degree in color classes of $G$.
 We show that the vertices of $U_1$ and $V_1$  are downer with respect to any non-zero eigenvalue $\la$ of $G$.
We may assume that $G$ has no isolated vertices. Let $u_1\in U_1$, so $N(u_1)=V_1\cup \cdots\cup V_k$.
Let $\x$ be any eigenvector for $\la$. We claim that $\x(u_1)\ne0$, from which the result follows.
For a contradiction, assume that  $\x(u_1)=0$. So, $\x$ is zero on the whole $U_1$. For any $v\in V_k$,  $N(v)=U_1$,  so by the sum rule, $\x(v)=0$.
Hence for any $u_2\in U_2$,
$$0=\la\x(u_1)=\sum_{v\in N(u_1)}\x(v)=\sum_{v\in V_1\cup\cdots\cup V_k}\x(v)=\sum_{v\in V_1\cup\cdots\cup V_{k-1}}\x(v)=\sum_{v\in N(u_2)}\x(v)=\la\x(u_2).$$
It follows that $\x$ is zero on $U_2$ as well.
For any $v\in V_{k-1}$,  $N(v)=U_1\cup U_2$,  so again by the sum rule, $\x(v)=0$.
Hence for any $u_3\in U_3$,
$$0=\la\x(u_1)=\sum_{v\in V_1\cup\cdots\cup V_k}\x(v)=\sum_{v\in V_1\cup\cdots\cup V_{k-2}}\x(v)=\sum_{v\in N(u_3)}\x(v)=\la\x(u_3).$$
It follows that $\x$ is zero on $U_3$, too.
Continuing this argument, it follows that $\x=\bf0$, a contradiction.
}\end{proof}

A chain graph for which $|U_1|=\cdots=|U_k|=|V_1|=\cdots=|V_k|=1$ is called a {\em half graph}, where we denote it by $H(k)$.
As we will see in what follows, specific half graphs provide counterexamples to Conjecture~\ref{conj}.
Let
$$(a_1,\ldots,a_6):=(1,0,-1,-1,0,1).$$
Let $$\x:=(x_1,\ldots,x_k)~\hbox{where}~ x_i=a_s~\hbox{if}~i\equiv s\hspace{-0.25cm}\pmod6.$$ In the next theorem, we show that the vector $(\x~\x)$ (each $\x$ corresponds to a color class) is an eigenvector of a non-zero eigenvalue of $H(k)$ for some $k$. In view of Remark~\ref{downer}, this disproves Conjecture~\ref{conj} .

\begin{thm}\label{1} In any half graph $H(k)$, the vector $(\x~\x)$ is an eigenvector for eigenvalue $1$ if $k\equiv1\pmod6$ and it is an eigenvector for eigenvalue $-1$ if $k\equiv4\pmod6$.
\end{thm}
\begin{proof}{
From Table~\ref{a_i}, we observe that for $1\le s\le6$,
$$\sum_{i=1}^{5-s}a_i=-a_s~~\hbox{and}~~\sum_{i=1}^{2-s}a_i=a_s,$$
where we consider $5-s$ and $2-s$ modulo 6 as elements of $\{1,\ldots,6\}.$
\begin{table}[ht]
$$\begin{array}{cccccc}
  \hline
   s&a_s & 5-s& \sum_{i=1}^{5-s}a_i  & 2-s&  \sum_{i=1}^{2-s}a_i  \\
   \hline
  1 & 1 & 4 & -1 & 1 & 1 \\
   2& 0 & 3 & 0 & 6 & 0 \\
  3 &-1 & 2 & 1 & 5 & -1 \\
  4 &-1 & 1 & 1 & 4 & -1 \\
  5 &0 & 6 & 0 & 3 & 0 \\
  6 & 1 & 5 & -1 & 2 & 1 \\
  \hline
\end{array}$$ \caption{The values of $\sum_{i=1}^{5-s}a_i$ and  $\sum_{i=1}^{2-s}a_i$}\label{a_i}
\end{table}

Note that, since $\sum_{i=1}^6a_i=0$, if $1\le\ell\le k$, $1\le s\le6$ and $\ell\equiv s\pmod6$, then
$$\sum_{i=1}^\ell x_i=\sum_{i=1}^s a_i.$$
Let $\{u_1,\ldots,u_k\}$ and $\{v_1,\ldots,v_k\}$ be the color classes of $H(k)$.
Let $k=6t+4$. We show that $(\x~\x)$ satisfies the sum rule with $\la=-1$.
By the symmetry, we only need to show this for $u_i$'s.
Let $i=6t'+s$ for some $1\le s\le6$. Then $n-i+1=6(t-t')+5-s$.
$$  \sum_{j:\,v_j\sim u_i}x_j=\sum_{j=1}^{n-i+1}x_j=\sum_{j=1}^{5-s}a_j=-a_s=-x_i.$$

Now, let $k=6t+1$. We show that in this case $(\x~\x)$ satisfies the sum rule with $\la=1$.
Let $i=6t'+s$ for some $1\le s\le6$. Then $n-i+1=6(t-t')+2-s$.
$$  \sum_{j:\,v_j\sim u_i}x_j=\sum_{j=1}^{n-i+1}x_j=\sum_{j=1}^{2-s}a_j=a_s=x_i.$$
}\end{proof}

Now we give another class of counterexamples to Conjecture~\ref{conj}. For this, let $$\om^2+\om-1=0,$$ and
$$(b_1,\ldots,b_{10}):=(\om,-1,0,1,-\om,-\om,1,0,-1,\om).$$
Let $$\x:=(x_1,\ldots,x_k)~\hbox{where}~ x_i=b_s~\hbox{if}~i\equiv s\hspace{-0.25cm}\pmod{10}.$$

\begin{thm}\label{om} In any half graph $H(k)$, the vector $(\x~\x)$ is an eigenvector for eigenvalue $\om$ if $k\equiv7\pmod{10}$ and it is an eigenvector for eigenvalue $-\om$ if $k\equiv2\pmod{10}$.
\end{thm}
\begin{proof}{From Table~\ref{b_i}, we observe that for $1\le s\le10$,
$$\sum_{i=1}^{8-s}b_i=\om b_s~~\hbox{and}~~\sum_{i=1}^{3-s}b_i=-\om b_s,$$
where we consider $8-s$ and $3-s$ modulo 10 as elements of $\{1,\ldots,10\}$.
\begin{table}[ht]
$${\small\begin{array}{cccccc}
\hline
 s &b_s &8-s &\sum_{i=1}^{8-s}b_i & 3-s&\sum_{i=1}^{3-s}b_i \\ \hline
1 & \om&7 &1-\om &2 & \om-1\\
 2 & -1& 6&-\om &1  &\om\\
 3 &0 & 5& 0 &10 &0\\
 4 &1 &4  &\om &9 &-\om\\
 5 & -\om& 3&\om-1 & 8& 1-\om\\
 6  &-\om &2 & \om-1&7 & 1-\om\\
 7  &1  &1 &\om &6 &-\om\\
 8  &0 &10 &0 &5 &0\\
 9  &-1 & 9 &-\om &4 &\om\\
  10 &\om &8 & 1-\om&3  &\om-1 \\ \hline
\end{array}}$$\caption{The values of $\sum_{i=1}^{8-s}b_i$ and  $\sum_{i=1}^{3-s}b_i $}\label{b_i}
\end{table}

Note that, since $\sum_{i=1}^{10}b_i=0$, if $1\le\ell\le k$, $1\le s\le10$ and $\ell\equiv s\pmod{10}$, then
$$\sum_{i=1}^\ell x_i=\sum_{i=1}^s b_i.$$
Let $k=10t+7$. We show that $(\x~\x)$ satisfies the sum rule with $\la=\om$.
Let $i=10t'+s$ for some $1\le s\le10$. Then $n-i+1=10(t-t')+8-s$.
$$  \sum_{j:\,v_j\sim u_i}x_j=\sum_{j=0}^{n-i+1}x_j=\sum_{j=1}^{8-s}b_j=\om b_s=\om x_i.$$
Now, let $k=10t+2$. Assume that $i=10t'+s$ for some $1\le s\le10$. Then $n-i+1=6(t-t')+3-s$.
$$  \sum_{j:\,v_j\sim u_i}x_j=\sum_{j=1}^{n-i+1}x_j=\sum_{j=1}^{3-s}b_j=-\om b_s=-\om x_i.$$
It follows that in this case $(\x~\x)$ satisfies the sum rule with $\la=-\om$.
}\end{proof}

\begin{remark} (i) Given $(\x,\x)$ as eigenvector of $H(k)$ for $\la\in\{\pm1,\pm\om\}$, then $(\x,-\x)$ is an eigenvector of $H(k)$ for $-\la$.
This gives more eigenvalues of $H(k)$ with eigenvectors containing zero components.
  (ii) Let $\x$ be an eigenvector for eigenvalue $\la$ of a graph $G$ with $\x(v)=0$ for some vertex $v$. If we add a new vertex $u$ duplicate to $v$ and add a zero component to $\x$ corresponding to $u$, then the new vector is an eigenvector of $H$ for eigenvalue $\la$.
   So, we can extend any graph presented in Theorems~\ref{1} or \ref{om} to construct infinitely many more counterexamples for Conjecture~\ref{conj}.
\end{remark}

\section{An eigenvalue-free interval }

In \cite{aas}, it was proved that chain graphs have no eigenvalues in the interval $(0,1/2)$ (and hence no eigenvalue in the interval $(-1/2,0)$, as the eigenvalues of bipartite graphs are symmetric with respect to zero). Here we give a simple proof for this result.

\begin{thm} {\rm(\cite{aas})} Chain graphs have no eigenvalue in the interval $(0,1/2)$.
\end{thm}
\begin{proof}{The proof goes by induction on the number of vertices. The assertion holds for bipartite graphs with at most 4 vertices (see \cite[p.~17]{bh}). It suffices to consider connected graphs. So let $G$ be a connected chain graph with at least 5 vertices.

First assume that $G$  has a pair of duplicates $u,v$ and $H=G-v$.
Let $\la_1\ge\cdots\ge\la_\ell$ and $\mu_1\ge\cdots\ge\mu_{\ell-1}$ be the eigenvalues of $G$ and $H$, respectively.
Also suppose that $\mu_t>\mu_{t+1}=\cdots=\mu_{t+j}=0>\mu_{t+j+1}$ (with possibly $j=0$).
By the induction hypothesis, $\mu_t>1/2$ (the equality is impossible).
By interlacing, we have $\la_{t+1}\ge0=\la_{t+2}=\cdots=\la_{t+j}=0\ge\la_{t+j+1}\ge\mu_{t+j+1}$.
Note that $\mul(0,G)=\mul(0,H)+1=j+1$. This is  possible only if both $\la_{t+1}$ and $\la_{t+j+1}$ are zero.
On the other hand, again by interlacing, $\la_t\ge\mu_t>1/2$. Hence $G$ has no eigenvalue in $(0,1/2)$.

Now, suppose that $G$ has no pair of duplicates.
 It follows that $G$ is a half graph and
 $$A(G)=\begin{pmatrix}O&C\\C^\top&O\end{pmatrix},$$
 with $C+C^\top=J_n+I_n$ where  $J_n$ is the all 1's $n\times n$ matrix.
 We have that
$$(2C-I)(2C-I)^\top=4CC^\top-2C-2C^\top+I=4CC^\top-I-2J.$$
This means that $4CC^\top-I=(2C-I)(2C-I)^\top+2J$ is positive semidefinite and so the eigenvalues of $CC^\top$ are not smaller than $1/4$. It turns out that $G$ has no eigenvalue in the interval $(-1/2,1/2)$. This completes the proof.
}\end{proof}

\section*{Acknowledgments}
The research of the author was in part supported by a grant from IPM.

{}

\end{document}